\documentclass[12pt,a4paper]{amsart}
\usepackage{amsfonts,amsmath,amsthm,amssymb, amscd}
\usepackage[top=2cm, bottom=4cm, left=3cm, right=1cm]{geometry}

\def\gp#1{\langle#1\rangle}

\newtheorem{theorem}{Theorem}
\newtheorem{lemma}{Lemma}

\title[Elements of high order]{Elements of high order in finite fields \\ specified by binomials}
\author[ Bovdi, Diene, Popovych]{Victor Bovdi, Adama Diene, Roman Popovych}

\address{
\textsl{\tiny V.~Bovdi}\newline
United Arab Emirates University, Al Ain, UAE}
\email{vbovdi@gmail.com}
\address{
\textsl{\tiny A.~Diene}\newline
United Arab Emirates University, Al Ain, UAE}
\email{adiene@uaeu.ac.ae}
\address{\textsl{\tiny R.~Popovych}\newline
Lviv Polytechnic National University,  Lviv,  Ukraine}
\email{rombp07@gmail.com}

\keywords{finite field, multiplicative order, binomial,  element of high multiplicative order}
\subjclass{11T30; 11Y99; 11T06; 20K01; 13M10}

\begin{document}

\maketitle

\begin{abstract}
Let $F_q$ be a field with $q$ elements, where  $q$ is a power of a  prime number  $p\geq 5$.
For any integer $m\geq 2$ and $a\in F_q^*$ such that the polynomial $x^m-a$ is irreducible in $F_q[x]$, we combine two different methods to construct explicitly elements of high order in the field $F_q[x]/\langle x^m-a\rangle $. Namely,   we find  elements with multiplicative order of at least $5^{\sqrt[3]{m/2}}$, which is better than previously  obtained bound for such family of extension fields.
\end{abstract}

\section*{Introduction}
The problem of efficiently constructing a primitive element for a given finite field is  notoriously difficult  computational task  of finite fields. That is why one considers a less restrictive question: namely  to find an element with "high" or "large" multiplicative order. Based on
\cite[Introduction, p.\,1615]{Gao}, by "large order" (high order, exponential order) of an element in the finite field $F_{q^{m}}$ of $q^{m}$ elements, we mean that the order of this element must be bigger than every polynomial in $\log (q^{m})$ when $q^{m} \rightarrow \infty$.  In general, we are not required to compute the exact order of such an element, but  it is sufficient  to obtain its lower bound.

High order elements in finite fields are very useful  in several applications, such as cryptography, coding theory, pseudo random number generation and combinatorics.

Gao  \cite{Gao} provided  an algorithm for constructing high order
elements for many (conjecturally all) general extensions $F_{q^{m}}$
of a finite field $F_{q}$ with lower bound  $\exp (\Omega((\log
m)^{2} / \log \log m))$ on the order. This bound was improved  in
\cite{Popovych_2014_b}. Voloch \cite{Voloch} proposed a method which
constructs an element of order of at least $\exp (\Omega((\log
m)^{2}))$. However, for some  classes of finite fields, it is
possible to construct elements of much   higher orders (for example,
see \cite{Ahmadi_Shparlinski_Voloch, Popovych_2012, Popovych_2014,
Popovych_Skuratovskii, Gathen_Shparlinski}). In these articles,
extensions   connected with cyclotomic polynomials  are considered
and elements of order  bounded by  $\exp (\Omega(\sqrt{m}))$ are
constructed. Note that this bound is much better than the ones   we
mentioned previously.

Some another classes of extensions based on the Kummer or Artin-Schreier polynomials were  considered in \cite{Brochero_Reis, Cheng, Popovych_2013}. The
best known lower bound of the order (see \cite[Theorem 1, p.\,87]{Popovych_2013})  for extensions, specified by Kummer polynomials, is
$2^{\lfloor\sqrt[3]{2 m}\rfloor}$, where $\lfloor\sqrt[3]{2 m}\rfloor$ is the highest integer less or equal to $\sqrt[3]{2 m}$.
In our current article we continue this line of investigation.

\section{Main  Results}

Let $p,q,m, n\in \mathbb{N}$ such that $p$ is an odd prime, $q=p^n$ and $m\geq 2$. Let $F_{q}$ be a finite field of $q$ elements and let  $a\in F_{q}^*$ such that  $x^{m}-a$ is an irreducible polynomial over $F_{q}$.  Let  $F_{q^{m}}=F_{q}(\theta)=F_{q}[x]/\gp{x^{m}-a}$ be a field extension of $F_{q}$ based on the irreducible  binomial (Kummer polynomial) $x^{m}-a$, where   $\gp{f(x)}$ is an ideal of $F_{q}[x]$ generated by  $f(x)\in F_{q}[x]$ and  $\theta$ is the coset of $x$ in $F_{q}[x]/\gp{x^{m}-a}$.

We widely use (see Lemma \ref{LEM:1}) the following fact from \cite{Panario_Thomson}. Let $F_{q}$ be a finite field of characteristic  $p \geq 5$. There exists infinitely many natural numbers $m$ and $a=a(m)\in F^*_{q}$, such that $x^m-a\in F_{q}[x]$ is  an irreducible polynomial. Such elements  $a=a(m)\in F^*_{q}$ are called {\it $m$-related}.

Our main results use  the fact   that the extension degree $m=k\cdot l$ is a product of two numbers, where $k>1$ is a divisor of $q-1$, and $l\geq 1$ is the order of the number $q$ modulo $m$. Using  a special representation of elements of the group $\gp{q \pmod{m}}$ (see \cite[Lemma 4, p.\,88]{Popovych_2013}) we deduce the following: if $q-1$ has a "large"  divisor $k$, we use for  the construction of the method similar to the case $F_{q}[x] /\gp{x^{m}-a}$ with the condition $q\equiv 1\pmod{m}$ or to the case $F_{p}[x] /\gp{x^{p}-x-a}$ (\cite[p.\, 363-365]{Cheng});
if $q-1$ does not have   a big divisor $k$, then $l=m / k$ is large, and we use for the construction of the method similar to the case $F_{q}[x] /\gp{x^{r-1}+\cdots+x+1}$ (see \cite{Ahmadi_Shparlinski_Voloch, Popovych_2012}). We take in both cases a linear binomial in some power of $\theta$ and all consecutive $q$-th powers of it (the so called conjugates), that also belong to the group generated by the binomial, and construct their distinct products.

In the first case, when $q \equiv 1\pmod{k}$, the conjugates of $\theta^{l}+b$ ($b\in F^*_{q}$) are linear binomials in $\theta^{l}$. The idea was introduced by Berrizbeitia \cite{Berrizbeitia} as an improvement of the AKS primality proving algorithm and developed by several authors (see \cite{Cheng,Popovych_2012} and also the survey  article \cite{Granville}). Our first result, which uses the first mentioned method, is the following.

\begin{theorem}\label{TEO:1}
Let $q=p^n$, where  $p\geq 5$ is a prime and let $m=k\cdot l\in \mathbb{N}$ such that  $k>1$ is a divisor of $q-1$, and $l\geq1$ is the multiplicative order of  $q \pmod{m}$.  Let $a\in F^*_{q}$ be an $m$-related element (i.e.  $x^{m}-a\in F_{q}[x]$ is  an irreducible polynomial)  and let the element $\theta$ define the field extension $F_{q}(\theta)=F_{q}[x] /\gp{x^{m}-a}$.

If  $b\in F^*_{q}$, then  the multiplicative order of $\theta^{l}+b\in F_{q}(\theta)$ is at least
\[
\mathfrak{d}_{1}:=\max _{0 \leq d_{-} \leq d<k}\Big\{\textstyle \binom{k}{d_{-}}\binom{d}{d_{-}}\binom{2k-d-d_{-}-1}{k-d-1}\Big\}.
\]

Moreover, if $k \geq 70$, then $\mathfrak{d}_{1} \geq 5^{k}$.
\end{theorem}
\smallskip

Note that  $\mathfrak{d}_{1} \geq \frac{5,7556^{k}}{30 k^{3 / 2}}$ for $k \geq 8$  by \cite[Theorem 1, p.\, 23, Corollary 2, p.\, 25]{Popovych_2015}).  It is easy to see that our  lower bound is better for  $k \geq 70 $.

Hence, we derive a lower bound on the order of $\theta^{l}+b$, which depends on $k$.

\bigskip

In the second case, the conjugates of $\theta+b$ are non-linear polynomials in $\theta$. The idea was introduced by Gathen and Shparlinski for the fields based on cyclotomic polynomials \cite{Gathen_Shparlinski}, and developed in \cite{Ahmadi_Shparlinski_Voloch, Popovych_2012, Popovych_2014}.

Let
\begin{equation}\label{Equ:1}
\begin{split}
\mathbf{T}=\{(u_{0},  \ldots, u_{l-1})\in \mathbb{Z}^{l}\mid \quad \textstyle \sum_{i=0}^{l-1}(i &\cdot k+1) u_{i}<m,\quad
0 \leq u_{0}, \ldots, u_{l-1} \leq p-1
\}.
\end{split}
\end{equation}

Our second result, which uses the second mentioned above method, is the following.

\begin{theorem}\label{TEO:2}
Let $q=p^n$, where  $p\geq 5$ is a prime and let $m=k\cdot l\in \mathbb{N}$ such that  $k>1$ is a divisor of $q-1$, and $l\geq 1$ is the multiplicative  order of  $q \pmod{m}$.  Let $a\in F^*_{q}$ be an $m$-related element (i.e.  $x^{m}-a\in F_{q}[x]$ is  an irreducible polynomial)  and let the element $\theta$ define the field extension $F_{q}(\theta)=F_{q}[x] /\gp{x^{m}-a}$.

If  $b\in F^*_{q}$, then the multiplicative order of  $\theta+b\in F_{q}(\theta)$  is at least $\mathfrak{d}_{2}:=|T|\geq 5^{\sqrt{l / 2}}$.

Moreover:
\begin{itemize}

\item[(i)] if $l \geq p^{2}+1$, then
\[
\textstyle \mathfrak{d}_{2} \geq\left(\frac{p(p-1)}{160(l-1)}\right)^{\sqrt{p}} \exp \left(2.5\cdot \sqrt{(1-\frac{1}{p})(l-1)}\right);
\]

\item[(ii)] if $l<p+1$, then
\[
\textstyle  \mathfrak{d}_{2} \geq \frac{\exp \big(2.5 \cdot\sqrt{l-1}\big)}{13(l-1)}.
\]
\end{itemize}
\end{theorem}

Items (i) and (ii) of Theorem \ref{TEO:2}  are slightly    better comparatively with $5^{\sqrt{l /2}}$ lower bounds on $\mathfrak{d}_{2}$, but not so explicit and not for all values of $l$. To prove these items we use Lemmas \ref{LEM:4},\ref{LEM:5}, and \ref{LEM:6}.  We take $\pi \sqrt{2 / 3} \approx 2.5$ to simplify formulas in (i) and (ii).

Hence, we obtain lower bounds on the order of the element $\theta+b$ which depend on $l$ or on $l$ and $p$.

Our third result, which takes together our first and second results, is the following.

\begin{theorem}\label{TEO:3}
Let $q=p^n$, where  $p\geq 5$ is a prime and let $m=k\cdot l\in \mathbb{N}$ such that  $k>1$ is a divisor of $q-1$, and $l\geq1$ is the multiplicative order of  $q \pmod{m}$.
Let $a\in F^*_{q}$ be an $m$-related element (i.e.  $x^{m}-a\in F_{q}[x]$ is  an irreducible polynomial)  and let the element $\theta$ define the field extension $F_{q}(\theta)=F_{q}[x] /\gp{x^{m}-a}$.

It is always possible to construct explicitly  in the field $F_{q}(\theta)$ an element of which the multiplicative order is at least $\max \{\mathfrak{d}_{1},\mathfrak{d}_{2}\}$.

Moreover, if $k \geq 70$ then the multiplicative order is at least $5^{\sqrt[3]{m / 2}}$.
\end{theorem}

The best previously known lower bound on the order of elements for finite field extensions defined by a binomial is equal to $2^{\sqrt[3]{2 m}}=2,3948^{\sqrt[3]{m}}$ (see \cite[Theorem 1, p.\,87]{Popovych_2013}). Our Theorem \ref{TEO:3} gives a new  bound $5^{\sqrt[3]{m / 2}}=3,5873^{\sqrt[3]{m}}$. which is an improvement of $2^{\sqrt[3]{2 m}}$.

\section{Lemmas and proofs}

The multiplicative group  $F^*_{q^m}$ of the finite field $F_{q^m}$ is  cyclic of order $q^m-1$  with  $\varphi(q^{m}-1)$  generators which are   called primitive elements, where $\varphi$ is the Euler totient function. For an element $g$ of a group $G$, we denote by $\langle g \rangle$ the cyclic subgroup generated by $g$.

Let $c$ be a fixed  positive integer. A partition $\mathcal{P}(c)$ of  $c$ is a sequence of  non-negative integers $u_{1}, \ldots, u_{c}$ such that
\begin{equation}\label{Equ:2}
\textstyle
c=\sum_{j=1}^{c} j u_{j}.
\end{equation}
We define  the following three numbers
\begin{equation}\label{Equ:3}
\mathfrak{u}(c),\quad  \mathfrak{u}(c, d), \quad \mathfrak{q}(c, d)
\end{equation}
related to  some subsets of the set of all partitions given by \eqref{Equ:2}:
\begin{itemize}
\item[($P_1$)] $\mathfrak{u}(c)$ is the  number of all  partitions  $\mathcal{P}(c)$;
\item[($P_2$)] $\mathfrak{u}(c, d)$ is the  number of those   partitions  $\mathcal{P}(c)$ for which $u_{1}, \ldots, u_{c} \leq d$ (i.e. each part appears no more than $d$ times);
\item[($P_3$)] $\mathfrak{q}(c, d)$ is the  number of those   partitions  $\mathcal{P}(c)$ for which $u_{j}=0$ if $j \equiv 0 \pmod{d}$, (i.e. each part of which is not divisible by $d$).
\end{itemize}

In a finite field of characteristic two, the polynomial $x^m-a$ is irreducible if and  only if $m=1$.
For a finite field  of an odd characteristic the question when the polynomial $x^m-a$ is irreducible
was done by Panario and Thomson \cite{Panario_Thomson}. In the case $p=3$ the only possible extension is for $m=2$, namely the irreducible polynomial $x^2 - 2$. If $p \geq 5$, then we can construct the extensions for infinitely many $m$.

\begin{lemma}\label{LEM:1}(see \cite[Theorem 2, p.\,3]{Panario_Thomson})
Let $F_{q}$ be a finite field of characteristic  $p \geq 5$.

For $m \not \equiv 0\pmod 4$ there exists an irreducible binomial over $F_{q}$ of degree $m$ if and only if every prime
factor of $m$ is also a prime factor of $q-1$.

For $m \equiv 0\pmod 4$  there exists an
irreducible binomial over $F_{q}$ of degree $m$ if and only if $q \equiv 1\pmod 4$ and every prime
factor of $m$ is also a prime factor of $q-1$.
\end{lemma}

Note that \cite{Panario_Thomson},  not only provides the possible degrees $m$ such that irreducible binomials $x^m-a$ exist,  but also provides a procedure   to construct the $m$-related elements $a=a(m)$.

\begin{lemma}\label{LEM:2} (see \cite[Lemma 4, p.\,88]{Popovych_2013})
Let  $m \geq 2$ and let $a\in F^*_{q}$ be an $m$-related element (i.e.  $x^{m}-a\in F_{q}[x]$ is  an irreducible polynomial).
If  $m=k\cdot l \in\mathbb{N}$, in which  $k$ is a divisor of $q-1$ and  $l$ is the order of $q$ modulo $m$,  then
$ \langle q\rangle\leq   \mathbb{Z}^{*}_{m}$ can be written as
\[
\langle q\rangle=\{\overline{i \cdot k+1} \mid i=0, \ldots, l-1\}.
\]
\end{lemma}

The next result below is a typical tool how to construct  high order elements (see \cite{Cheng, Gao, Handbook}).

\begin{lemma}\label{LEM:3}
Let  $m \geq 2$  and let $f(x)\in F_{q}[x]$ be an  irreducible polynomial  of degree $m$.
Let $g(x), h(x)\in F_{q}[x]$ such that  $g(x)\not=h(x)$.  If  $\deg(g(x))$ and $\deg(h(x))$ are less than  $m$, then
\[
g(x)+ \gp{f(x)}\not=g(x)+ \gp{f(x)}\in F_{q}[x] /\gp{f(x)}.
\]
\end{lemma}

\begin{lemma}\label{LEM:4}(Glaisher, 1883, see \cite[Corollary 1.3, p.\,6.]{Andrews})
The number of partitions $\mathcal{P}(n_0)$ of $n_0\in\mathbb{N}$ not containing $d_0$ equal parts is equal to the number of partitions $\mathcal{P}(n_0)$ of $n_0$ with no part divisible by $d_0$, i.e.
\[
\mathfrak{u}(n_0, d_0-1)=\mathfrak{q}(n_0, d_0).
\]
\end{lemma}

\begin{lemma}\label{LEM:5}(\cite[Theorem 5.1]{Maroti})
For all integers $d_0 > 1$ and $n_0\geq d_0^2$, we have
\[
\textstyle (\frac{d_0(d_0-1)}{160 n_0})^{\sqrt{d_0}} \exp \Big(2.5 \cdot \sqrt{(1-\frac{1}{d_0}) n_0}\Big)<\mathfrak{q}(n_0, d_0).
\]
\end{lemma}

\begin{lemma}\label{LEM:6}(\cite[Theorem 4.2]{Maroti})
For all integers $n_0 >1$
\[
\textstyle \frac{\exp \Big(2.5 \cdot \sqrt{n_0}\Big)}{13n_0}<\mathfrak{u}(n_0).
\]
\end{lemma}

\begin{proof} [\underline{Proof of Theorem \ref{TEO:1}}]
Since $k$ is a divisor of $q-1$, then, according to Lemma \ref{LEM:1}, the binomial $y^{k}-a$ is irreducible over $F_{q}[y]$. Set $\eta=\theta^{l}$. Clearly,  $\eta^{k}=\theta^{m}=a$,   and  $F_{q}(\eta)=F_{q}[y] /\gp{y^{k}-a}$ is a subfield of $F_{q}(\theta)$.

Consider $\eta+b$ (a linear binomial in the power of $\eta=\theta^{l}$) and consequential $q$-th powers (conjugates) of it, that belong to the group generated by this binomial. Write $q-1=k h$ for some integer $h$. Then $\eta^{q}+b=(\eta^{k})^{h} \eta+b=a^{h} \eta+b$ and the conjugates of $\eta+b$ are equal to
\[
(\eta+b)^{q^{i}}=a^{h i} \eta+b \qquad\qquad  (i=0, \ldots, k-1).
\]
Consider the subgroup $H=\gp{a^{h i} \eta+b\mid i=0, \ldots, k-1} \leq \gp{\eta+b} \leq F_{q}^*(\eta)$.
For a vector $\alpha=(u_{0},\ldots, u_{k-1})\in \mathbb{Z}^k$, we define the  product
\begin{equation}\label{Equ:4}
\mathrm{P}(\alpha)=\prod_{i=0}^{k-1}(a^{h i} \eta+b)^{u_{i}}\in H.
\end{equation}

The next  combinatorial problem  was introduced by Voloch in order to improve the AKS primality proving algorithm and this method has been developed by several authors (see surveys \cite[p.\,31--32]{Granville} and \cite{Popovych_2015}).

This   problem consists of finding for a fix  $k\in \mathbb{N}$ two non-negative integers  $0 \leq d_{-} \leq d<k$  with maximal possible value of  the    product $\binom{k}{d_{-}}\binom{d}{d_{-}}\binom{k-d_{-}-d-1}{k-d-1}$ of the following three binomial coefficients $\binom{k}{d_{-}}, \binom{d}{d_{-}}$ and $\binom{k-d_{-}-d-1}{k-d-1}$.

It is easy to check that  this product $\binom{k}{d_{-}}\binom{d}{d_{-}}\binom{k-d_{-}-d-1}{k-d-1}$ is the cardinality of  the  set  $S=\{ (u_{0}, u_{1}, \ldots, u_{k-1})\in\mathbb{Z}^k\}$ with the following properties:
\begin{itemize}
\item[(i)] the number of negative components equals $d_{-}$;
\item[(ii)] the sum of absolute values of negative components $\sum_{i, u_{i}<0}\left|u_{i}\right| \leq d$;
\item[(iii)] the sum of positive components $\sum_{i, u_{i} \geq 0} u_{i} \leq k-1-d$,
\end{itemize}
in which  $0 \leq d_{-} \leq d<k$.

Indeed, to specify the element of this set, we choose at first places, where vector values are negative: this takes into account the factor $\binom{k}{d_{-}}$. Then we choose values of  negative elements so that the sum of their absolute values does not exceed $d$: this takes into account the factor $\binom{d}{d_{-}}$. Finally,  we choose  non-negative vector values at $k-d_{-}$ places, so that their sum does not exceed $k-1-d$: this takes into account the factor $\binom{k-d_{-}+k-1-d}{k-1-d}$.

For each  $(u_{0}, u_{1}, \ldots, u_{k-1})\in S$ we consider the product \eqref{Equ:4} and  claim that two different vectors $(u_{0}, u_{1}, \ldots, u_{k-1})$ and $(v_{0}, v_{1}, \ldots, v_{k-1})$ from  $S$ give different values of $P$.

Let   $\alpha=(u_{0},\ldots, u_{k-1}),\beta=(v_{0},\ldots, v_{k-1})\in S$ such that $\alpha\not=\beta$ and $P(\alpha)=P(\beta)$ (see \eqref{Equ:4}).
Since  $y^{k}-a\in F_q[y]$ is the characteristic polynomial of $\eta$, then
\[
\prod_{i=0}^{k-1}(a^{h i} y+b)^{u_{i}}\equiv \prod_{i=0}^{k-1}(a^{h i} y+b)^{v_{i}}\pmod {\gp{y^{k}-a}}
\]
and, as a consequence,
\begin{equation}\label{Equ:5}
\begin{split}
    f_1(y)\equiv f_2(y)\pmod{\gp{y^{k}-a}},
\end{split}
\end{equation}
where
\[
\begin{split}
f_1(y):&=  \prod_{\substack{0 \leq i \leq k-1, \\0\leq u_i }}(a^{h i} y+b)^{u_{i}} \prod_{\substack{0 \leq i \leq k-1,\\ 0> v_i}}(a^{h i} y+b)^{|v_{i}|};\\
f_2(y):&=\prod_{\substack{0 \leq i \leq k-1, \\u_i<0}}(a^{h i} y+b)^{|u_{i}|} \prod_{\substack{0 \leq i \leq k-1,\\ v_i \geq 0}}(a^{h i} y+b)^{v_{i}}.
\end{split}
\]
Using \eqref{Equ:5}  and the  facts that
\[
\begin{split}
\deg(f_1(y))&=\sum_{\substack{0 \leq i \leq k-1, \\u_i \geq 0}} u_i+\sum_{\substack{0 \leq i \leq k-1,\\ v_i<0}}|v_i| \leq(k-1-d)+d=k-1,\\
\deg(f_2(y))&=\sum_{\substack{0 \leq i \leq k-1, \\ u_i \leq 0}}|u_i|+\sum_{\substack{0 \leq i \leq k-1, \\ v_i \geq  0}} v_i \leq d+(k-1-d)=k-1,
\end{split}
\]
we conclude  that $f_1(y)= f_2(y)$ from    Lemma \ref{LEM:3}.

Moreover, each factor  $a^{hi} y+b$ in $f_1(y)$ ($=f_2(y)$)  is irreducible and $a^{hi} y\not=a^{hj} y$ for $i\not=j$.
Since  $F_{q}[y]$ is a unique factorization ring, we obtain a contradiction.

Hence, the number of $\alpha \in S$ such that $P(\alpha)\in H$ (see \eqref{Equ:4}) is equal to the cardinality of $S$. We choose $d_{-}$ and  $d$ to obtain maximum of elements in $S$. As a result, $\eta+b$ has the multiplicative order at least $\mathfrak{d}_{1}$.
\end{proof}

\begin{proof}[\underline{Proof of Theorem \ref{TEO:2}}]
According to Lemma \ref{LEM:2}, for each $z \in\{0, \ldots, l-1\}$  there exist  unique $i\in\{0, \ldots, l-1\}$ and $j=j({i})\in\mathbb{Z}$, such that   $q^{z} = (i \cdot k+1)+j \cdot m$. Then conjugates of element $\theta+b$ are equal to
\[
(\theta+b)^{q^{z}}=\theta^{q^{z}}+b=(\theta^{m})^{j}\theta^{i \cdot k+1}+b=a^{j} \theta^{i \cdot k+1}+b\in \gp{\theta+b}.
\]
Similarly, as in the proof of Theorem \ref{TEO:1} (see also \eqref{Equ:1} before Theorem \ref{TEO:2}), for a vector $\alpha=(u_{0}, \ldots, u_{l-1})\in\mathbf{T}$  we define the  product
\[
\textstyle
\mathrm{P}(\alpha)=\prod_{i=0}^{l-1} (a^{j}\theta^{i k+1}+b)^{u_{i}}\in \gp{\theta+b}.
\]
We claim that if $\beta=(v_{0},\ldots, v_{l-1})\in\mathbf{T}$ is  distinct from $\alpha\in\mathbf{T}$, then $P(\alpha)\not=P(\beta)$.

Indeed, let $P(\alpha)=P(\beta)$. Set
\[
f_1(x):= \prod_{i=0}^{l-1}(a^{j} x^{i \cdot k+1}+b)^{u_{i}}\in F_q[x]\quad \text{ and }\quad f_2(x):=\prod_{i=0}^{l-1}(a^{j} x^{i\cdot k+1}+b)^{v_{i}}\in F_q[x].
\]
Clearly, $\deg(f_1(x))=\sum_{i=0}^{1-1}(i k+1) u_{i}<m$ and $\deg(f_1(x))=\sum_{i=0}^{1-1}(i k+1) u_{i}<m$.
Since  $x^{m}-a$ is the characteristic polynomial of $\theta$, then
\[
\begin{split}
    f_1(x)\equiv f_2(x)\pmod{\gp{x^{m}-a}},
\end{split}
\]
so $f_1(x)= f_2(x)$ by Lemma \ref{LEM:3}.

Note that $F_{q}[x]$ is a unique factorization ring. Let $r$ be the smallest integer for which $u_{r} \neq v_{r}$ and, say $u_{r}>v_{r}$, in which $u_{i}\in\alpha$ and $v_{i}\in \beta$. After removing common factors on both sides of the equation $f_1(x)= f_2(x)$, we observe  that
\begin{equation}\label{Equ:6}
(a^{j_{r}} x^{r\cdot k+1}+b)^{u_{r}-v_{r}} \prod_{i \geq r+1}^{l-1}(a^{j_{i}} x^{i \cdot k+1}+b)^{u_{i}}=\prod_{i \geq r+1}^{l-1}(a^{j_{i}} x^{i r+1}+b)^{v_{i}}.
\end{equation}
The absolute term for the polynomial $\prod_{i \geq r+1}^{l-1}(a^{j_{i}} x^{i\cdot k+1}+b)^{u_{i}}$ we denote by $c$. Then there is the term $(u_{r}-v_{r}) a^{j_{r}} b^{u_{r}-v_{r}-1} c x^{r\cdot k+1}$ in the polynomial on the left side of \eqref{Equ:6}  with the minimal non-zero power of $x$. Since $0 \leq u_{r}, v_{r} \leq p-1, u_{r} \neq v_{r}, a, b, c \neq 0$, the term is non-zero. This  term does not occur on the right side, which makes the identity \eqref{Equ:6} impossible. Hence, products, corresponding to distinct solutions, cannot be equal. Consequently, the multiplicative order of $\theta+b$ in $F_{q}(\theta)=F_{q}[x] /\gp{x^{m}-a}$  is at least $\mathfrak{d}_{2}:=|T|$.

Let $\tau\in [2, p-1]$ be an integer. Let us choose the largest integer   $\alpha>0$ such that
\[
\sum_{i=0}^{\alpha}(i \cdot k+1)(\tau-1)<m, \qquad\qquad  (k \geq 2).
\]
Obviously,
\[
\sum_{i=0}^{\alpha}(i \cdot k+1)(\tau-1)=\textstyle \frac{(\tau-1)(\alpha k+2)(\alpha+1)}{2}<\frac{(\tau-1) k(\alpha+1)^{2}} {2}.
\]
If  $\alpha:=\lfloor\sqrt{\frac{2 l}{\tau-1}}-1\rfloor$, then $(\tau-1) k(\alpha+1)^{2}\leq 2 m$ and for integers
\[
u_{i} \in\begin{cases}[0, \tau-1] & \text{for \qquad } i=0, \ldots, \alpha;\\
0&\text{for \qquad } i=\alpha+1, \ldots, l-1,\\
\end{cases}
\]
the  vector $(u_0,\ldots,u_{l-1})$ belongs to the  set $\mathrm{T}$. The number of such vectors is $\tau^{\alpha+1}\leq \tau ^{\sqrt{\frac{2 l}{\tau-1}}}$.
To choose $\tau$ we investigate  the maximum value of the following function:
\[
\textstyle
f(\tau):=\tau^{ \sqrt{\frac{2 l}{\tau-1}}}=\exp \Big\{ \sqrt{\frac{2 l}{\tau-1}} \cdot \ln (\tau)\Big\}, \qquad (2 \leq \tau \leq p-1).
\]
Obviously, $f^{\prime}(\tau)= \tau^{\sqrt{\frac{2 l}{\tau-1}}} \cdot \Big(\frac{2 l}{\tau-1}\Big)^{\frac{1}{2}} \cdot\Big(\frac{1}{\tau}-\frac{\ln \tau}{2(\tau-1)}\Big)$ and our function $f(\tau)$ reaches  the  maximum value at the point $\tau_{0}\in (4,92155, 4,921555)$.  Moreover,  $f(\tau)$ monotonically decreases for $\tau \geq 5\geq \lceil \tau_{0}\rceil$, so $\mathfrak{d}_{2}=|\mathrm{T}|\geq 5^{\sqrt{l/ 2}}$.

(i) Let us show that    $\mathfrak{d}_{2} \geq \mathfrak{u}(l-1, p-1)$ (see \eqref{Equ:3} and ($P_2$)). Indeed $i k+1<k(i+1)$, so
\begin{equation}\label{Equ:7}
\textstyle
\sum_{i=0}^{l-1}(ki+1) u_{i}<k \sum_{i=0}^{l-1}(i+1) u_{i}<m,
\end{equation}
and\qquad   $\sum_{i=1}^{l}i u_{i-1}<\textstyle\frac{m}{k}=l$.

If  $u_{l-1}\not=0$ we obtain a contradiction. Hence  $u_{l-1}=0$ and $\sum_{i=1}^{l-1} i u_{i-1}=l-1$, so  $(u_0,\ldots,u_{l-2})$ is a partition of $l-1$ (see \eqref{Equ:2}) such that  $0 \leq u_{0}, \ldots, u_{l-2} \leq p-1$ and \eqref{Equ:7} holds.

Explicit lower bounds on $\mathfrak{q}(n_0, d_0)$ for $n \geq d_0^{2}$ and on $\mathfrak{u}(n_0)$ for all $n_0\in\mathbb{Z}$ are  given in \cite{Maroti}. Note that $\mathfrak{u}(n_0, d_0-1)=\mathfrak{u}(n_0)$ for  $n_0<d_0$. Using   Lemmas \ref{LEM:4} and \ref{LEM:5} for $n_0:=l-1$ and $d_0:=p-1$ we obtain  that
\[
\begin{split}
\textstyle \mathfrak{d}_{2} \geq \mathfrak{u}(l-1, p-1)&=\mathfrak{q}(l-1, p)\\
&>\textstyle \big(\frac{p(p-1)}{160(l-1)}\big)^{\sqrt{p}} \textstyle  \exp\Big(2.5 \cdot \textstyle \sqrt{(1-\frac{1}{p})(l-1)}\Big).
\end{split}
\]
(ii) Recall that if $n_0<d_0$, then $\mathfrak{u}(n_0, d_0-1)=\mathfrak{u}(n_0)$. Consequently
\[
\begin{split}
\textstyle \mathfrak{d}_{2} \geq \mathfrak{u}(l-1, p-1)&=\mathfrak{u}(l-1)>\textstyle  \frac{\exp \big(2.5 \cdot \sqrt{l-1}\big)}{13(l-1)},
\end{split}
\]
where $n_0=l-1$ and $d_0=p$ by   Lemmas \ref{LEM:5} and \ref{LEM:6}.
\end{proof}

\begin{proof}[\underline{Proof of Corollary \ref{TEO:3}}] Elements   $\theta^{l}+b, \theta+b\in F_{q}(\theta)$ are different and  their  orders are  at least $\mathfrak{d}_{1}$ and  $\mathfrak{d}_{2}$, respectively by Theorems \ref{TEO:1} and \ref{TEO:2}.

If ${k} \leq {\sqrt{l / 2}}$, then  the order of  $\theta+b$ has  a lower bound $5^{\sqrt{l / 2}}$  by  Theorem  \ref{TEO:1}. If ${k} > {\sqrt{l / 2}}$, then we construct  the element $\gamma=\theta^{m_{2}}+b$ with lower bound $5^{k}$ on its order by Theorem  \ref{TEO:2}. Hence, one can explicitly construct in the field $F_{q}[x] /\gp{x^{m}-a}$ an element with the multiplicative order of at least
$\max\{5^{k}, 5^{\sqrt{l / 2}}\}$. In the worst case these lower bounds are equal: $5^{k}=5^{\sqrt{l / 2}}$. Then $k=\sqrt[3]{m^{2} / 2}$ and the order is at least $5^{\sqrt[3]{m / 2}}$.
\end{proof}

Note that the two considered methods have two parts: the algebraic part and the combinatorial calculation. An improvement in either (or both) of these parts results in an improvement in the evaluation of the method and then in the approach.
Generalizing the approach to other classes of finite fields is an open problem.

\section*{Acknowledgement}
\noindent
The work was supported by the UAEU UPAR grant  G00003431.

\end{document}